\documentclass[12pt]{article}

\usepackage[utf8]{inputenc}

\usepackage[english]{babel}

\usepackage{amsmath,amsfonts,amssymb,amsthm}

\usepackage[left=1.5cm,right=1.5cm,top=1.5cm,bottom=1.5cm]{geometry}

\usepackage{tikz}
\usetikzlibrary{shapes.geometric}
\usepackage{wrapfig}

\usepackage[pdfencoding=unicode, psdextra, citecolor=blue, urlcolor=blue,
linkcolor=blue, colorlinks=true, bookmarksopen=true]{hyperref}

\usepackage{comment}

\usepackage{bbm}

\newcommand{\R}{\mathbb{R}}
\newcommand{\id}{\mathbbm{1}}

\DeclareMathOperator{\conv}{conv}
\DeclareMathOperator{\aff}{aff}

\newtheorem{theorem}{Theorem}

\newtheorem{lemma}{Lemma}
\newtheorem{corollary}{Corollary}

\theoremstyle{definition}
\newtheorem{definition}{Definition}

\theoremstyle{remark}
\newtheorem{example}{Example}
\newtheorem{remark}{Remark}
\newtheorem{question}{Question}

\title{Finsler and sub-Finsler geodesics with chattering.}
\author{L.V.~Lokutsievskiy, M.I.~Zelikin}

\begin{document}

\maketitle

\begin{abstract}

	In this paper, we provide examples of Finsler and sub-Finsler manifolds whose geodesics exhibit chattering, that is, a countable number of switches over an arbitrarily small time interval. We also present an explicit left-invariant structure on a Carnot group whose geodesics exhibit chattering. This provides a negative answer to Le Donne's question. Furthermore, the paper presents a sufficient condition for normal Pontryagin maximum principle extremals in (sub-)Finsler problems to exhibit chattering.
	
\end{abstract}

\section{Introduction}
In optimal control theory, chattering (or the Fuller phenomenon, named after its discoverer, \cite{Fuller}) is frequently observed. It is characterized by a countable number of optimal control switches within arbitrarily small time intervals. In the context of optimal control problems with drift, this phenomenon has been thoroughly studied as the junction of second-order singular extremals with non-singular ones \cite{ZB}. However, this phenomenon has not been observed in sub-Riemannian problems, even though sub-Riemannian geometry possesses a wealth of singular extremals. We believe this stems not from the absence of drift in sub-Riemannian geometry, but from the strict convexity of the norm. In the Riemannian and sub-Riemannian cases, the unit ball in the tangent space is an ellipsoid, and therefore its boundary is strictly convex. On Finsler and sub-Finsler manifolds, the norms on tangent spaces are not necessarily Euclidean, and the unit balls can be, for example, convex polyhedra.

This paper aims to demonstrate, in the simplest possible manner, explicit examples of Finsler and sub-Finsler manifolds in which shortest paths (and, consequently, geodesics) exhibit chattering. In these examples, the absence of strict convexity of the norm will play a key role. The crucial point is that the chattering phenomenon arises at the junction of second-order singular extremals with non-singular ones. Consequently, if the norm is strictly convex, normal Pontryagin maximum principle extremals can never be singular, hence such a junction is impossible. Therefore, it is natural to consider the class of Finsler and sub-Finsler problems that contain singular extremals that are not abnormal. In other words, the norm is not strictly convex.

\medskip

In this paper, we present:

\begin{itemize}
	\item an example of a Finsler manifold (see \S~\ref{sec: finsl example});
	\item an example of a sub-Finsler manifold (see \S~\ref{sec: subf example});
	\item an example of a Carnot group (see \S\ref{sec: Carnot}) and two left-invariant structures on it: a Finsler structure (see \S\ref{subsec: Carnot Finsler}) and a sub-Finsler structure (see \S\ref{subsec: Carnot sub Finsler}),
\end{itemize}
such that in each of these examples, there exists a pair of points connected by a unique shortest path, and this shortest path contains chattering that is identical to the chattering in the classical Fuller problem. So this provides a negative answer to Le Donne's question (see~\cite[Question 6.1.15]{LeDonne}). Additionally, in these examples, we provide a sufficient condition for a pair of points to be connected by a unique shortest path, where the shortest path contains chattering.

In \S\ref{sec: general theorem}, we present a theorem that establishes a sufficient condition for the existence of chattering extremals in Finsler and sub-Finsler optimal control problems.

In \S\ref{sec: chaos}, we address the emergence of another phenomenon in Finsler and sub-Finsler geometry related to optimal control problems with drift, where the sequence of control switches in chattering between vertices (e.g., of a triangle) is determined by a topological Markov chain.

Basic definitions related to Finsler and sub-Finsler polyhedral structures are provided in \S\ref{sec: subfinsler def} and \S\ref{sec: subf distance}.

A brief overview of the classical Fuller problem and the chattering phenomenon within it is given in~\S\ref{sec: chattering}.

\section{Two definitions of a polyhedral sub-Finsler structure}
\label{sec: subfinsler def}

Let $M$, $\dim M=n$, be a smooth manifold. A Finsler structure on $M$ is a family of norms on $T_xM$ (a good introduction to Finsler geometry is given in~\cite{Bao}). Using this family of norms, one can measure the lengths of curves and thus turn any Finsler manifold into a metric space. The structure is called polyhedral if the unit balls of the indicated norms are polyhedra.

The well-known Weyl–Minkowski states that any compact convex polyhedron can be defined in two equivalent ways: as the convex hull of its vertices, and as the compact intersection of a finite set of half-spaces. In the case of a polyhedral Finsler structure, we obtain not a single polyhedron but a whole family, and, hence, generally two different definitions of a polyhedral sub-Finsler structure. Let us give both of these definitions.

\begin{definition}[Direct]
	\label{defn: polyhedral1}
	
	Let $f_1,\ldots,f_N$ be (smooth) vector fields on $M$ such that for every point $x\in M$ it holds\footnote{Here $\mathrm{rint}$ denotes the relative interior of a convex set.} $0\in\mathrm{rint}\,B_1(x)$, where $B_1(x)=\conv(f_1(x),\ldots,f_{N}(x))$. Then the polyhedral sub-Finsler structure is defined by the subnorm via the Minkowski functional\footnote{As usual, we take $\inf\emptyset = +\infty$.}:
	$$
	\forall \xi\in T_x M\qquad \|\xi\|_1=\mu_{B_1(x)}(\xi) = \inf\{a\ge 0: \xi\in aB_1(x)\}
	\in \R\cup\{+\infty\}.
	$$
\end{definition}

Note that, generally speaking, $\|\xi\|_1\ne \|-\xi\|_1$. Moreover, for some $\xi\in T_xM$ it may happen that $\|\xi\|_1=+\infty$, if the linear span of the vectors $f_i(x)$ does not coincide with the entire tangent space $T_xM$. Therefore, it is convenient to introduce the notation
$$
\Delta(x) = \mathrm{span}\,\{f_1(x),\ldots,f_N(x)\} \subset T_xM.
$$
In particular, if $\xi\in T_xM\setminus \Delta(x)$, then $\|\xi\|_1=+\infty$. If $\Delta(x) = T_xM$ for all $x$, the structure is usually called \textit{Finsler}. Otherwise, the structure is called \textit{sub-Finsler}. Therefore, the following dual description (by the Weyl–Minkowski theorem) contains two types of 1-forms: one set defines the subspace $\Delta(x)$, and the second defines the Finsler norm within it.

\begin{definition}[Dual]
	\label{defn: polyhedral2}
	
	Let $\lambda_1,\ldots,\lambda_J$ and $\zeta_1,\ldots,\zeta_I$ be (smooth) 1-forms on $M$. Denote
	$$
	\Delta(x) = \bigcap_{i=1}^I\ker \zeta_i.
	$$
	Then if for every $x\in M$ the set
	$$
	B_2(x) = \{\xi\in \Delta(x):\ \forall j\ \lambda_j[\xi]\le 1\}\subset T_xM
	$$
	is compact, then the polyhedral sub-Finsler structure is defined by the subnorm
	$$
	\forall \xi\in T_x M\qquad \|\xi\|_2=\mu_{B_2(x)}(\xi)\in \R\cup\{+\infty\}.
	$$
	
\end{definition}

The second formula can be conveniently rewritten as:
\begin{equation}
	\label{eq: norm 2 explicit}
	\|\xi\|_2 = \begin{cases}
		\max_j \lambda_j[\xi],&\text{if }\xi\in\Delta(x);\\
		+\infty,&\text{if }\xi\not\in\Delta(x).
	\end{cases}
\end{equation}

\bigskip

Generally speaking, definitions~\ref{defn: polyhedral1} and~\ref{defn: polyhedral2} are not equivalent, for the following reason. As $x$ changes, both balls $B_1(x)$ and $B_2(x)$ can behave non-smoothly, however the ball $B_1(x)$ behaves similarly to the maximum of smooth functions, while $B_2(x)$ behaves similarly to the minimum.

\begin{example}
	
	Let $M=\R$, $f_1(x)=-1$, $f_2(x)=1$ and $f_3(x)=\frac12(1+x^2)$. Then the ball $B_1(x)$ has the form
	$$
	B_1(x) = \big[-1;\,\max\{1,\tfrac12(1+x^2)\}\big]
	$$
	and in a neighborhood of points $x=\pm1$ it cannot be defined using smooth 1-forms $\lambda_j$ in definition~\ref{defn: polyhedral2}.
	
	Conversely, if $\lambda_1(x)=-1$, $\lambda_2(x)=1$ and $\lambda_3(x)=\frac2{1+x^2}$, then the ball $B_2(x)$ has the form
	$$
	B_2(x) = \big[-1;\,\min\{1,\tfrac12(1+x^2)\}\big]
	$$
	and in a neighborhood of points $x=\pm1$ it cannot be defined using smooth vector fields $f_j$ in definition~\ref{defn: polyhedral1}.
	
\end{example}
\section{Distance on a sub-Finsler manifold}
\label{sec: subf distance}

Let $\|\cdot\|$ denote either $\|\cdot\|_1$ or $\|\cdot\|_2$ from now on. The distance between points $x_0$ and $x_1$ on $M$ is defined as the infimum of the lengths of (absolutely continuous, or equivalently, Lipschitz) curves connecting them:
$$
d_{SF}(x_0,x_1) = \inf\left\{
\mathrm{length}(x(\cdot))=\int_0^1 \|\dot x(t)\|\,dt\ \Big|\  x(0)=x_0,x(1)=x_1
\right\}.
$$
Equivalently, the distance between points can be defined as the minimum time of travel along naturally parameterized curves:
$$
d_{SF}(x_0,x_1) = \inf\left\{
T\ \big|\ \|\dot x(t)\|\le 1,x(0)=x_0,x(T)=x_1
\right\}.
$$

Let us now formulate two optimal control problems that are equivalent to finding shortest path. A good introduction to the sub-Finsler geometry from optimal control point of view is given in~\cite{BBLDS}. An interesting method for explicit finding sub-Finsler geodesics for the case $\dim\Delta=2$ is suggested in \cite{LokutsievskiyCT, LokutsievskiyCT2}.

The first definition is convenient in the case $\|\cdot\|=\|\cdot\|_2$, since the subnorm $\|\cdot\|_2$ can be computed explicitly by formula~\eqref{eq: norm 2 explicit}. Indeed, if we denote $\Delta(x)=\bigcap_i\ker \zeta_i(x)$, then we obtain an optimal control problem
$$
\int_0^1 \max_j\lambda_j(\dot x)\,dt\to\min
$$
$$
\dot x\in \Delta(x);
$$
$$
x(0)=x_0;\qquad x(1)=x_1.
$$
A slight inconvenience here is that the linear subspace $\Delta(x)$ is defined as the intersection of the kernels of 1-forms. However, this obstacle is easily overcome: if $g_1,\ldots,g_K$ are vector fields on $M$ such that $\Delta(x)=\mathrm{span}\, g_k(x)$ (such vector fields always exist), then $\dot x=\sum v_kg_k(x)$, where the controls $v_k$ are unbounded, $v_k\in\R$.

The second definition is conveniently reformulated as an optimal control problem (specifically: time-optimal) in the case $\|\cdot\|=\|\cdot\|_1$:
\begin{equation}
	\label{eq: time minimiztion problem}
	\left\{\begin{array}{c}
		T\to\min\\
		\dot x = \sum_{i}u_if_i(x);\\
		u_i\ge 0,\ \sum_i u_i=1;\\
		x(0)=x_0,\ x(1)=x_1.
	\end{array}\right.
\end{equation}
For this problem, the Pontryagin Maximum Principle takes a particularly convenient form:
\begin{equation}
	\label{eq: H for time minimiztion}
	H = \sum_{j}u_j\langle p, f_j(x) \rangle \to \max_{u_i\ge 0;\ \sum_iu_i=1}.
\end{equation}
Since the Pontryagin function $H$ is linear in the controls, the maximum is always attained on some face.

\section{Chattering in the Fuller problem.}
\label{sec: chattering}

The aim of this work is to demonstrate the presence of chattering in polyhedral Finsler and sub-Finsler structures. Therefore, in this section, we briefly present the classical Fuller problem, in which all optimal curves (except the identically zero one) exhibit chattering. This problem is not sub-Finsler (as it contains drift), but it will be useful for our further discussion.

\begin{example}[Fuller\footnote{The factor $\frac12$ here is not essential and is added for historical reasons.}, \cite{Fuller}]
	\begin{equation}
		\label{problem Fuller}
		\begin{gathered}
			\frac12\int_0^{+\infty} x^2\,dt \to\min\\
			\dot x = y;\quad \dot y = u;\quad |u|\le 1\\
			x(0) = x_0;\quad y(0)=y_0.
		\end{gathered}
	\end{equation}
	Here $(x(t),y(t))$ is a curve in the plane $\R^2$, $u(t)\in[-1;1]$ is the control, and $(x_0,y_0)$ is the initial point.
\end{example}

Pontryagin Maximum Principle (PMP) for this problem gives the following Pontryagin's function, $\lambda_0\in\{0,1\}$:
\begin{equation}
\label{eq: PMP for Fuller Problem}
	H = -\frac{\lambda_0}2 x^2 + py + qu.
\end{equation}
It is easy to see that $\lambda_0\ne 0$ on any optimal trajectory. Indeed, if $\lambda_0=0$, then $q(t)$ is a quadratic polynomial. Consequently $x(t)$ is a quadratic polynomial for large enough $t$ and hence $\int_0^\infty x^2\,dt=\infty$. So $\lambda_0=1$.

It is well known that for any initial point $(x_0,y_0)$, there exists a unique optimal trajectory $(x(t),y(t),u(t))$. This trajectory reaches the origin at some time $0\le T<\infty$, i.e., $x(t)=y(t)=u(t)=0$ for all $t\ge T$. Moreover, if $(x_0,y_0)\ne (0,0)$, then as $t\to T-0$, the optimal control $u(t)$ performs a countable number of switches between $1$ and $-1$. This phenomenon is called chattering (or the Fuller phenomenon, in honor of its discoverer). Let us formulate this result as a theorem

\begin{theorem}[{see~\cite[Chapter 2]{ZB}}]
	\label{thm: fuller problem}
	Let $\mu$ denote the unique root of the equation $\mu^4-3\mu^3-4\mu^2-3\mu+1=0$ on the interval\footnote{$\mu=\frac14(3+\sqrt{33}-\sqrt{26+6\sqrt{33}})$} $\mu\in(0;1)$. Then for any initial point $(x_0,y_0)$, there exists a unique optimal trajectory $(x(t),y(t),u(t))$ in the Fuller problem~\eqref{problem Fuller}. Moreover, for this trajectory:
	\begin{enumerate}
		
		\item there exists\footnote{The index $F$ is used to indicate the Fuller problem.} $0\le T_F(x_0,y_0)<\infty$ such that $x(t)\equiv y(t)\equiv u(t)\equiv 0$ for $t\ge T_F(x_0,y_0)$, but $|u(t)|=1$ for almost all $t\in[0;T_F(x_0;y_0)]$;
		
		\item if $(x_0;y_0)\ne 0$, then there exists a number $\tau\in(0;T_F(x_0;y_0))$ such that the control $u(t)$ changes sign at points of the form $t_k=T_F(x_0;y_0)-\tau\mu^k$, $k=0,1,2,\ldots$, and is constant in between
		
	\end{enumerate}
	
\end{theorem}

Thus, if $(x_0,y_0)\ne (0,0)$, the control performs a countable number of switches on any arbitrarily small time interval $t\in(T_F(x_0;y_0)-\varepsilon;T_F(x_0,y_0))$, $\varepsilon>0$ (see Fig.~\ref{fig: fuller optimal synthesis}).

\begin{figure}
	\centering
	\includegraphics[width=0.4\linewidth]{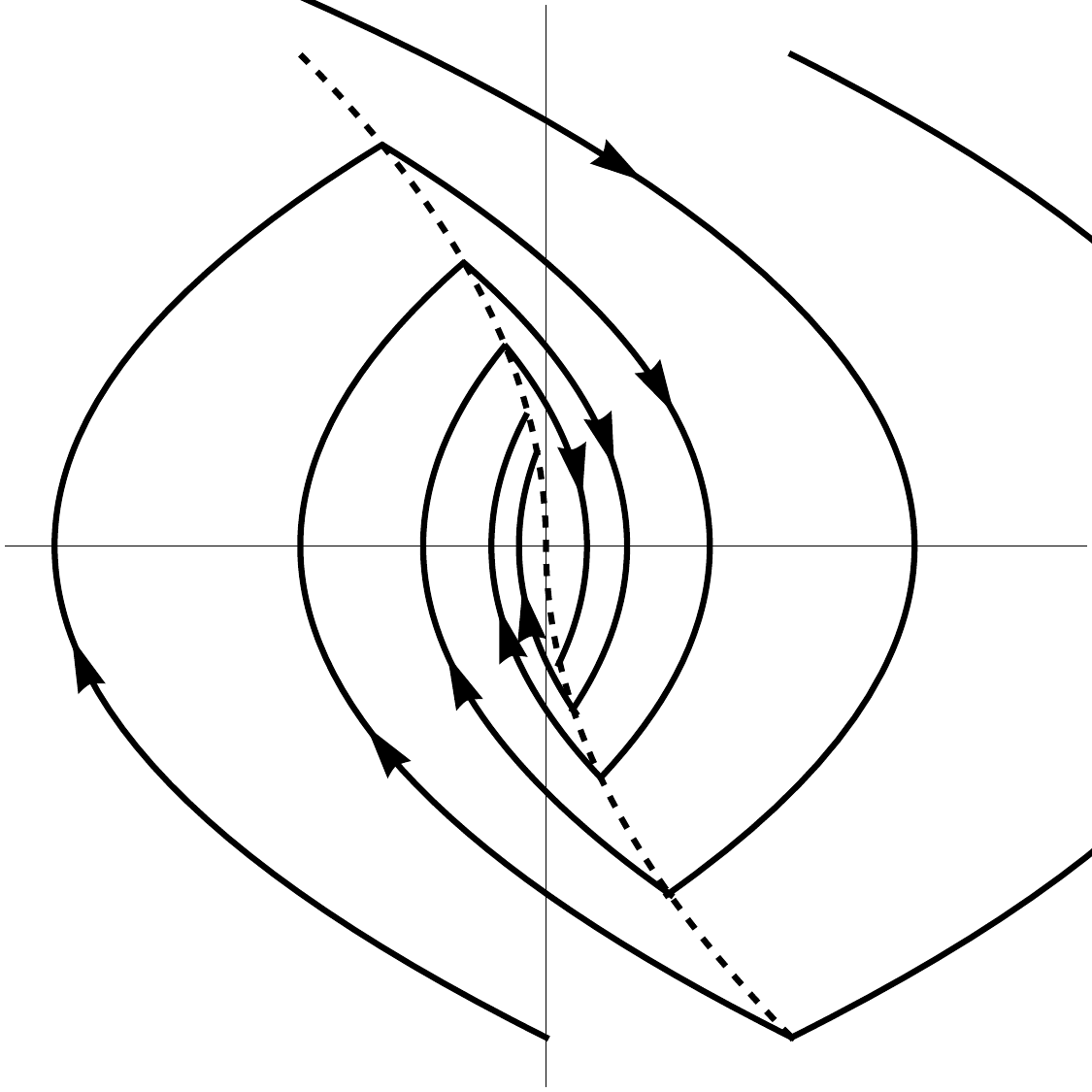}
	\caption{Optimal synthesis in the Fuller problem.}
	\label{fig: fuller optimal synthesis}
\end{figure}

We do not provide the proof, but we mention that a key fact is the existence of the Fuller symmetry group: for any $\lambda>0$ the following map
$$
x\mapsto \lambda^2x;\quad y\mapsto \lambda y;\quad t\mapsto \lambda t
$$
preserves the control system $\dot x=y$, $|\dot y|\le 1$ and changes the functional by a factor of $\lambda^5$. Therefore, such a map transforms the optimal trajectory from the point $(x_0,y_0)$ into the optimal trajectory from the point $(\lambda^2x_0,\lambda y_0)$. Thus, the Fuller symmetry group preserves the optimal synthesis on the plane $\R^2=\{(x,y)\}$. The proof of Theorem~\ref{thm: fuller problem} can be carried out using this symmetry. It is sufficient to show that every non-trivial optimal trajectory $(x(t),y(t))$ is self-similar, i.e., there exists a time moment $s>0$ such that $x(s+\mu t)=\mu^2 x(t)$ and $y(s+\mu t) = \mu y(t)$ for some $\mu\in(0;1)$ (in fact, the $\mu$ from Theorem~\ref{thm: fuller problem}). From this, it follows that $T_F(x_0,y_0)=s+\mu s+\mu^2 s+\ldots=s/(1-\mu)$, there is the same number of switches on each such interval, and consequently, there must be a countable number of them along the entire trajectory. For the details of the proof, we refer to the book~\cite{ZB}.

In the Fuller problem~\eqref{problem Fuller}, the exact formula is known both for the time to reach the origin $T_F(x_0,y_0)$ (see~\cite[Proposition~3.5]{ZMH}), and for the optimal value of the functional $J_F(x_0,y_0)$ (see~\cite[Proposition~3.2]{ZMH}), where
$$
J_F(x_0,y_0)  =  \inf \left\{
\frac12\int_0^\infty x^2\,dt
\ \Big|\
x(0)=x_0,\dot x(0)=y_0\text{ and }|\ddot x|\le 1
\right\}.
$$
We will not need these formulas. However, the asymptotics of these functions are very useful and quite easy to find. Indeed, it can be shown that the functions $T_F$ and $J_F$ are continuous and respect the symmetry group:
$$
\forall\lambda>0\quad T_F(\lambda^2 x_0,\lambda y_0) = \lambda T_F(x_0,y_0)\text{ and }J_F(\lambda^2 x_0,\lambda y_0) = \lambda^5 J_F(x_0,y_0).
$$
Therefore $T_F(x_0,y_0)\sim (x_0^2+y_0^4)^{\frac14}$ and $J_F(x_0,y_0)\sim (x_0^2+y_0^4)^{\frac54}$, where the estimation constants are exactly the maximum and minimum values of the functions $T_F(x_0,y_0)$ and $J_F(x_0,y_0)$ on the oval $x_0^2+y_0^4=1$.

It will be important for us that
$$
(x_0,y_0)\ne (0,0)
\qquad\Longrightarrow\qquad
0<T_F(x_0,y_0)<\infty
\text{ and }
0<J_F(x_0,y_0)<\infty.
$$

In Finsler and sub-Finsler problems, the control occurs over a finite time interval. Therefore, the structure of optimal trajectories in the Fuller problem over a finite time interval will be important to us.

Let $(x_0,y_0)\in\R^2$, $(x_1,y_1)\in\R^2$ and $t_1>0$. Consider the problem
\begin{equation}
	\label{problem: Fuller finite time}
	\begin{gathered}
		\frac12\int_0^{t_1} x^2\,dt \to\min\\
		\dot x = y;\quad \dot y = u;\quad |u|\le 1\\
		x(0) = x_0;\quad y(0)=y_0;\quad x(t_1) = x_1;\quad y(t_1) = y_1.
	\end{gathered}
\end{equation}
\begin{lemma}
	\label{lm: Fuller finite time}
	If $t_1\ge T_F(x_0,y_0) + T_F(x_1,-y_1)$, then in the problem~\eqref{problem: Fuller finite time} the optimal trajectory $x_F(t),y_F(t),u_F(t)$ exists, is unique, and is constructed as follows. Let $x^{0}(t)$, $y^{0}(t)$, $u^{0}(t)$ and $x^{1}(t)$, $y^{1}(t)$, $u^{1}(t)$ denote the optimal trajectories in the classical Fuller problem~\eqref{problem Fuller} starting from points $(x_0,y_0)$ and $(x_1,-y_1)$, respectively. Then
	\begin{itemize}
		\item $x_F(t)=x^0(t)$, $y_F(t)=y^0(t)$ and $u_F(t)=u^0(t)$ for $t\in [0;T_F(x_0,y_0)]$;
		\item $x_F(t)=y_F(t)=u_F(t)=0$ for $t\in [T_F(x_0,y_0);t_1-T_F(x_1,-y_1)]$;
		\item $x_F(t)=x^1(t_1-t)$, $y_F(t)=-y^1(t_1-t)$ and $u_F(t)=u^1(t_1-t)$ for $t\in [t_1-T_F(x_1,-y_1);t_1]$.
	\end{itemize}
	In particular, if $(x_0,y_0)\ne (0,0)$ and $(x_1,y_1)\ne (0,0)$, then the optimal trajectory exhibits chattering both as $t\to T_F(x_0,y_0)-0$ and as $t\to t_1-T_F(x_1,-y_1)+0$, and is equal to 0 between these time moments.
\end{lemma}

\begin{proof}
	
	First, note that the specified trajectory is admissible. Indeed, $x_F(t)$ and $y_F(t)$ are Lipschitz, satisfy the boundary conditions, and the system of equations $\dot x_F=y_F$, $\dot y_F=u_F$ with $u_F\in[-1;1]$. PMP for this problem has the same Pontryagin function~\eqref{eq: PMP for Fuller Problem}. The trajectory $(x_F(t),y_F(t),u_F(t))$ satisfies the Pontryagin Maximum Principle with $\lambda_0=1$ and the corresponding costate variables $(p_F(t), q_F(t))$, which are obtained from the costate variables of the trajectories $(x^{0}(t),y^{0}(t),u^{0}(t))$ and $(x^{1}(t),y^{1}(t),u^{1}(t))$ similarly to the construction of the trajectory itself. Thus, the specified trajectory satisfies the PMP with $\lambda_0=1$ and is therefore optimal, since the problem is convex, and for convex problems, the PMP with $\lambda_0=1$ is a sufficient condition for optimality. Indeed, if $(x(t),y(t),u(t))$ is any other trajectory, then due to convexity we have:
	$$
	\frac12\int_0^{t_1} (x^2 - x_F^2)\,dt \ge
	\int_0^{t_1} x_F(x-x_F) \,dt =
	(p_F(x-x_F))\big|_0^{t_1}-\int_0^{t_1} p_F(y-y_F) \,dt
	$$
	The first term on the right-hand side is $0$, since $x(0)=x_F(0)=x_0$ and $x(t_1)=x_F(t_1)=x_1$. Therefore
	$$
	\frac12\int_0^{t_1} (x^2 - x_F^2)\,dt \ge
	-\int_0^{t_1} p_F(y-y_F) \,dt=
	-(q_F(y-y_F))\big|_0^{t_1} + \int_0^{t_1} q_F(u-u_F) \,dt \ge 0.
	$$
	Indeed, the first term is 0, since $y(0)=y_F(0)=y_0$ and $y(t_1)=y_F(t_1)=y_1$, and the second term is non-negative, since $q_F(t)u_F(t)\ge q_F(t)v$ for any $v\in[-1;1]$ and almost all $t$ according to the PMP. Uniqueness also easily follows from the strict convexity of $x^2$.
\end{proof}

\section{An explicit example of sub-Finsler geodesics with chattering}
\label{sec: subf example}
Let $M=\R^4=\{q=(x,y,z,w)\}$. Let us define a sub-Finsler structure on $M$ using two key vector fields:
\begin{equation}
	\label{eq: f1 f2 main field}
	f_1(q) = \begin{pmatrix}
		y\\
		0\\
		\frac12x^2\\
		1
	\end{pmatrix}
	\qquad
	f_2(q) =\begin{pmatrix}
		0\\
		1\\
		0\\
		0
	\end{pmatrix}
\end{equation}
It is easy to check that $\mathrm{Lie}\,(f_1,f_2)(q) = T_qM$ for any point $q\in M$. Indeed, their non-trivial brackets are structured as follows:
\begin{equation}
	\label{eq: f_i Lie brackets}
	[f_1,f_2] = f_3\ne0;\quad [f_1,f_3] = f_4\ne0;\quad [f_1,f_4]=f_5\ne0,\quad [f_2,f_5]=-[f_3,f_4]=f_6\ne0,
\end{equation}
The remaining brackets of $f_i$, $i\le 6$, independent of those listed, are equal to 0. It is easy to check that $\mathrm{span}(f_1,f_2,f_3,f_6)(q)=\R^4$ for any point $q$. Note that the vector fields $f_i$, $i=1,\ldots,6$ are obviously linearly dependent in the space $M=\R^4$. Indeed, direct calculation gives $f_4=xf_6$ and $f_5=yf_6$.
So, let us define the unit ball in the tangent space $T_qM$ as a quadrilateral with vertices $f_1+f_2$, $f_1-f_2$, $f_2-f_1$, and $-f_1-f_2$:
$$
B(q) = \conv (\pm f_1(q)\pm f_2(q)).
$$
Thus, we have defined a polyhedral sub-Finsler structure in the sense of the direct definition. Since $\mathrm{Lie}(f_1,f_2)(q) = \R^4$ for every point $q$, by the Rashevsky–Chow theorem the system is controllable, i.e.,
$$
\forall q_0,q_1\in\R^4\qquad\mathrm{dist}(q_0,q_1) < +\infty.
$$
This structure can also be easily defined dually. To do this, let us compute
$$
\Delta(q) = \mathrm{span}(f_1(q),f_2(q)) = \{uf_1(q) + vf_2(q)\,|\,(u,v)\in\R^2\}
$$
Therefore
$$
\Delta(q) = \ker\zeta_1 \cap\ker\zeta_2.
$$
where
$$
\zeta_1 = dx - y\,dw
\text{ and }
\zeta_2 = dz - \frac12x^2\,dw
$$
Moreover,
$$
B(q) = \big\{\xi\in \Delta(q): \lambda_j[\xi]\le 1, j=1,2,3,4 \big\}
$$
where
$$
\lambda_1 = dy;\quad \lambda_2 = -dy;\quad \lambda_3=dw;\quad \lambda_4 = -dw.
$$
In particular,
$$
\|\xi\| = \begin{cases}
	+\infty,&\text{if }dx[\xi]\ne y\,dw[\xi]\text{ or }dz[\xi]\ne \frac12x^2dw[\xi]\\
	\max\{|dy[\xi]|,|dw[\xi]|\},&\text{otherwise}.
\end{cases}
$$
It is easy to verify that by Filippov's theorem, for any two points, there exists a shortest path connecting them.
\begin{remark}
	The described sub-Finsler structure has the following convenient Fuller symmetry group: for any $\lambda>0$ the map
	$$
	x\mapsto \lambda^2 x;\quad y\mapsto \lambda y;\quad z\mapsto \lambda^5z;\quad w\mapsto \lambda w
	$$
	maps $f_1\mapsto \lambda f_1$ and $f_2\mapsto\lambda f_2$. Therefore, it scales the sub-Finsler metric by a factor of $\lambda$ and, consequently, maps geodesics to geodesics. The geodesics with chattering described below are self-similar with respect to this group if we add translations $z\mapsto z+c_1$, $w\mapsto w+c_2$.
\end{remark}
Let us now show that many geodesics in this sub-Finsler problem exhibit chattering. Fix $q_0,q_1\in\R^4$ and find the shortest path connecting points $q_0$ and $q_1$. As mentioned earlier, the shortest path is a solution to the following time-optimal control problem:
\begin{equation}
	\label{eq: subfinsler fuller control system}
	T\to\min;
	\qquad
	\dot q=vf_1(q) + uf_2(q)\ \Leftrightarrow\
	\begin{cases}
		\dot x = y v,\\
		\dot y = u,\\
		\dot z = \frac12x^2 v,\\
		\dot w = v;
	\end{cases}
	\quad
	|u|\le 1;\ |v|\le 1.
\end{equation}
\begin{equation}
	\label{eq: subfinsler fuller terminal constraints}
	\begin{array}{c}
		x(0) = x_0,\quad y(0) = y_0;\quad z(0) = z_0;\quad w(0) = w_0\\
		x(T) = x_1,\quad y(T) = y_1;\quad z(T) = z_1;\quad w(T) = w_1.
	\end{array}
\end{equation}
Here $q=(x,y,z,w)\in M=\R^4$ is the state variable, and $(u,v)\in\R^2$ is the control.
\begin{theorem}
	\label{thm: subf chattering}
	Assume\footnote{Here $T_F$ and $J_F$ are functions from the Fuller problem, defined in section \S\ref{sec: chattering}.} that $z_1-z_0=J_F(x_0,y_0) + J_F(x_1,-y_1)$ and $w_1-w_0\ge T_F(x_0,y_0) + T_F(x_1,-y_1)$. Let $(x_F(t),y_F(t))$ with control $u_F(t)$ be an optimal trajectory in the problem~\eqref{problem: Fuller finite time} with $t_1=w_1-w_0$. Put
	$$
		z_F(t) = z_0+\frac12\int_0^t x_F^2(s)\,ds,\quad
		w_F(t) = w_0+t,\text{ and }v_F(t)\equiv 1.
	$$
	Then the trajectory $(x_F(t),y_F(t),z_F(t),w_F(t))$ with controls $(u_F(t),v_F(t))$, $t\in[0;t_1]$, is the unique shortest path in the sub-Finsler problem~\eqref{eq: subfinsler fuller control system} connecting the points~\eqref{eq: subfinsler fuller terminal constraints}.
\end{theorem}
In particular, it immediately follows from this theorem that if $(x_0,y_0)\ne (0,0)$ and $(x_1,y_1)\ne (0,0)$, then by Lemma~\ref{lm: Fuller finite time} the shortest path connecting the points~\eqref{eq: subfinsler fuller terminal constraints} is unique and has two chattering regimes.
\begin{proof}
	The specified trajectory satisfies the control system~\eqref{eq: subfinsler fuller control system} and the boundary conditions~\eqref{eq: subfinsler fuller terminal constraints}. Thus, the sub-Finsler distance between these points is certainly no greater than $t_1$. Let us show that it is the unique optimal solution to the time-optimal control problem.
	First, we claim that if some trajectory of the control system~\eqref{eq: subfinsler fuller control system} satisfies the boundary conditions~\eqref{eq: subfinsler fuller terminal constraints}, then the motion time $T$ satisfies the estimate $T\ge t_1$; and if $T = t_1$, then $v(t)=1$ for almost all $t$. Indeed, since $|v|\le 1$, we have
	$$
	t_1 = w(T)-w(0) = \int_0^T \dot w\,dt = \int_0^T v\,dt \le T.
	$$
	and equality is possible only if $v(t)=1$ for a.e. $t$.
	Thus, the trajectory defined in the theorem reaches the endpoint exactly at time $T=t_1$ and is therefore optimal.
	Now let us show that the Fuller trajectory is the unique shortest path connecting the points~\eqref{eq: subfinsler fuller terminal constraints}. Indeed, if some trajectory $(x(t),y(t),z(t),w(t))$ with controls $(u(t),v(t))$ is optimal in the system~\eqref{eq: subfinsler fuller control system}, \eqref{eq: subfinsler fuller terminal constraints}, then the motion time $T$ coincides with $t_1$. Therefore, $v(t)\equiv v_F(t)\equiv 1$. Consequently, the trajectory $(x(t),y(t))$ with control $u(t)$ is optimal in the modified Fuller problem~\eqref{problem: Fuller finite time} on the finite interval. Therefore, by virtue of Lemma~\ref{lm: Fuller finite time}, we have $u(t)\equiv u_F(t)$ and, consequently, the curve $(x(t),y(t),z(t),w(t))$ coincides with the curve specified in the statement of the theorem.
\end{proof}

\section{An explicit example of Finsler geodesics with chattering}
\label{sec: finsl example}

To construct an example of a Finsler manifold with chattering, it is sufficient to slightly modify the example from the previous section. Let the vector fields $f_1$ and $f_2$ in $\R^4$ be given by relation~\eqref{eq: f1 f2 main field}. Then from the commutation relations~\eqref{eq: f_i Lie brackets}, we find
\begin{equation}
	\label{eq: f3 f4 f5 f6}
	f_3(q) = \begin{pmatrix}
		-1\\
		0\\
		0\\
		0
	\end{pmatrix};
	\quad
	f_4(q) = xf_6(q);\quad
	f_5(q) = yf_6(q);\quad
	f_6(q) =\begin{pmatrix}
		0\\
		0\\
		1\\
		0
	\end{pmatrix}.
\end{equation}
Consider the norm with the unit ball in the form of a hyperoctahedron
$$
B(q) = \conv(\pm f_1(q)\pm f_2(q);\pm f_3(q);\pm f_6(q))
$$
with 8 vertices $\pm(f_1(q)+f_2(q))$, $\pm(f_1(q)-f_2(q))$, $\pm f_3(q)$, $\pm f_6(q)$.

It is straightforward to compute the Finsler norm with such a unit ball. Indeed, if $\xi\in T_qM$, then
$$
\xi= \eta_1 f_3(q) + \eta_2 f_2(q) + \eta_3 f_6(q) + \eta_4 f_1(q)
\quad\Longrightarrow\quad
\|\xi\| = \max(|\eta_2|,|\eta_4|) + |\eta_1| + |\eta_3|.
$$
If we express $\eta_i$ in terms of the original coordinates $\xi=(\xi_1,\xi_2,\xi_3,\xi_4)$, we get
\begin{equation}
	\label{eq: explicit Finsler norm}
	\|\xi\| = \max(|\xi_2|,|\xi_4|) + |\xi_1-y\xi_4| + |\xi_3-\tfrac12x^2\xi_4|.
\end{equation}

We claim that shortest paths in this Finsler metric often exhibit chattering. Let us fix the initial and final points $q_{0,1}=({x_{0,1},y_{0,1},z_{0,1},w_{0,1}})$.

\begin{theorem}
	Assume\footnote{Here $T_F$ and $J_F$ are functions from the Fuller problem, defined in section \S\ref{sec: chattering}.} that $z_1-z_0=J_F(x_0,y_0) + J_F(x_1,-y_1)$ and $w_1-w_0\ge T_F(x_0,y_0) + T_F(x_1,-y_1)$. Let $(x_F(t),y_F(t))$ be an optimal trajectory in the problem~\eqref{problem: Fuller finite time} with $t_1=w_1-w_0$. Put
	$$
		z_F(t) = z_0+\frac12\int_0^t x_F^2(s)\,ds,\quad
		w_F(t) = w_0+t.
	$$
	Then the trajectory $(x_F(t),y_F(t),z_F(t),w_F(t))$, $t\in[0;t_1]$, is the unique shortest path in the Finsler problem with norm~\eqref{eq: explicit Finsler norm}, connecting the points~\eqref{eq: subfinsler fuller terminal constraints}.
\end{theorem}

Again, it immediately follows from this theorem that if $(x_0,y_0)\ne (0,0)$ and $(x_1,y_1)\ne (0,0)$, then by Lemma~\ref{lm: Fuller finite time} the shortest path connecting the points~\eqref{eq: subfinsler fuller terminal constraints} is unique and has two chattering regimes.

\begin{proof}
	
	To find the shortest paths, let us write down the corresponding time-optimal control problem~\eqref{eq: time minimiztion problem}. Since the unit ball $B(q)$ has 8 vertices $\pm f_1(q)\pm f_2(q)$, $\pm f_3(q)$ and $\pm f_6(q)$, the time-optimal control problem~\eqref{eq: time minimiztion problem} takes the form
	$$
	T\to\min
	$$
	$$
	\begin{aligned}
		\dot q =& u_1 (f_1(q) + f_2(q)) + u_2(f_1(q)-f_2(q)) + u_3(-f_1(q)+f_2(q)) + u_4(-f_1(q)-f_2(q))+\\
		&+ u_5 f_3(q) - u_6 f_3(q) + u_7 f_6(q) - u_8 f_6(q)
	\end{aligned}
	$$
	$$
	u_i\ge 0;\ \sum_i u_i = 1
	$$
	with boundary conditions \eqref{eq: subfinsler fuller terminal constraints}.
	
	Note that the trajectory $q_F(t) = (x_F(t),y_F(t),z_F(t),w_F(t))$ specified in the theorem has a natural parameterization, since on it $u_i(t)\equiv0$ for $i\ge 3$, $u_1(t)+u_2(t)\equiv 1$, and from the equality $u_1(t)-u_2(t)=u_F(t)\in[-1;1]$ it follows that $u_1(t)\ge 0$ and $u_2(t)\ge 0$. Consequently, the distance between the points~\eqref{eq: subfinsler fuller terminal constraints} is certainly no greater than the travel time along the curve $q_F(t)$, i.e., $t_1$.
	
	Let us show that the distance is indeed equal to $t_1$ and there are no other shortest paths connecting the points~\eqref{eq: subfinsler fuller terminal constraints}. The reasoning almost verbatim repeats the proof of Theorem~\ref{thm: subf chattering}. Indeed, if some trajectory connects the points~\eqref{eq: subfinsler fuller terminal constraints} in time $T\ge 0$, then
	$$
	t_1 = w(T)-w(0) = \int_0^T \dot w\,dt = \int_0^T u_1 + u_2 - u_3 - u_4\,dt \le T
	$$
	since $u_3\ge 0$, $u_4\ge 0$ and $u_1+u_2\le \sum_i u_i = 1$. Therefore, we get $u_1+u_2\equiv 1$ and $u_i\equiv0$ for $i=3,4$. Consequently, $u_i\equiv0$ also for $i=5,6,7,8$. If we denote $u=u_1-u_2$, we find that on any trajectory connecting the points~\eqref{eq: subfinsler fuller terminal constraints}, we have $\dot q = f_1(q) + uf_2(q)$ and $u\in[-1;1]$. It remains to use Lemma~\ref{lm: Fuller finite time} analogously to the proof of Theorem~\ref{thm: subf chattering}.
	
\end{proof}

\section{Chattering on Carnot groups}
\label{sec: Carnot}

In this section, we construct an explicit example of a Carnot group of step $5$, on which shortest paths exhibit chattering. We note that the explicit form of geodesics in left-invariant Finsler and sub-Finsler problems has been studied in a large number of works (see, e.g.,~\cite{ArdentovSachkov,LokutsievskiyCT,LokutsievskiyCT2,BBLDS}), but chattering has not been detected in them. The reason is that the Carnot group presented below has step $5$, whereas explicit formulas for geodesics were sought only on groups of step $2$ and $3$. Let us start with the construction of the group itself.

\subsection{Construction of a 6-dimensional Carnot group}

Let us define vector fields on $\R^6=\{x=(x_1,\ldots,x_6)\}$
$$
g_1 = \begin{pmatrix}1\\0\\-x_2\\-x_3\\-x_1x_3\\\frac12x_3^2 \end{pmatrix}
\quad
g_2 = \begin{pmatrix}0\\1\\0\\0\\0\\0 \end{pmatrix}
\quad
g_3 = \begin{pmatrix}0\\0\\1\\0\\0\\0 \end{pmatrix}
\quad
g_4 = \begin{pmatrix}0\\0\\0\\1\\x_1\\-x_3 \end{pmatrix}
\quad
g_5 = \begin{pmatrix}0\\0\\0\\0\\1\\x_2 \end{pmatrix}
\quad
g_6 = \begin{pmatrix}0\\0\\0\\0\\0\\1 \end{pmatrix}
$$
\begin{wrapfigure}{r}{4cm}
	\centering\begin{tikzpicture}
		\node (g1) at (0,0) {$g_1$};
		\node (g2) at (2,0) {$g_2$};
		\node (g3) at (1,-1) {$g_3$};
		\node (g4) at (1,-2) {$g_4$};
		\node (g5) at (1,-3) {$g_5$};
		\node (g6) at (2,-4) {$g_6$};
		
		\draw[->] (g1) -- (g3);
		\draw[->] (g2) -- (g3);
		\draw[->] (g1) -- (g4);
		\draw[->] (g3) -- (g4);
		\draw[->] (g1) -- (g5);
		\draw[->] (g4) -- (g5);
		\draw[->] (g2) -- (g6);
		\draw[->] (g5) -- (g6);
		
	\end{tikzpicture}
	\caption{Structure of brackets of vector fields $g_i$.}
	\label{fig: brackets graph}
\end{wrapfigure}
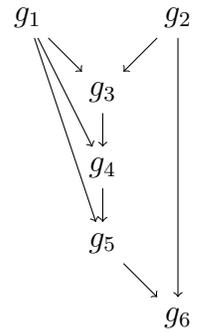
The idea behind choosing such vector fields is as follows: the vector fields $f_1$ and $f_2$ from paragraph~\ref{sec: subf example} generate a graded nilpotent Lie algebra over $\R$, but over the module of smooth functions their non-trivial brackets $f_i$, $i=1,\ldots,6$, are linearly dependent (see \eqref{eq: f3 f4 f5 f6}). The vector fields $g_i$ are chosen such that they are linearly independent at every point (and, therefore, over the module of smooth functions) and form a Lie algebra over $\R$ with the same commutation table~\eqref{eq: f_i Lie brackets} (see Fig.~\ref{fig: brackets graph}). Indeed, direct verification shows that the vector fields $g_i$ form a graded nilpotent Lie algebra of step 5 with the following non-trivial brackets:
$$
[g_1,g_2] = g_3;\quad [g_1,g_3] = g_4;\quad [g_1,g_4]=g_5,\quad [g_2,g_5]=-[g_3,g_4]=g_6.
$$

Moreover, any linear combination of these vector fields with constant coefficients is a complete vector field, so $\R^6$ admits exponential coordinates (let's place the start of motion at the origin $0_6\in\R^6$ for convenience)\footnote{Here $\mathrm{Exp}(f)(x_0)$ denotes the solution of the ODE $\dot x=f(x)$ with the initial condition $x(0)=x_0$ at time $t=1$, $\mathrm{Exp}(f)(x_0)\stackrel{def}{=}x(1)$.}:
$$
x=\mathrm{Exp}\left(\sum_{i=1}^6 a_ig_i\right)(0_6),
$$
From which, solving a simple system of ODEs, we find
$$
x_1=a_1,\
x_2=a_2,\
x_3=a_3 - \frac{1}{2} a_1 a_2,\
x_4=a_4 + \frac{1}{6} a_1^2 a_2 - \frac{1}{2} a_1 a_3,
$$
$$
x_5=a_5+\frac{1}{8} a_1^3 a_2  -\frac{1}{3} a_1^2a_3   +\frac{1}{2} a_1a_4,\
x_6=a_6+\frac{1}{40}a_1^3 a_2^2-\frac{1}{8} a_1^2a_2a_3+\frac{1}{6} a_1(a_3^2+a_2a_4)-\frac{1}{2}a_3 a_4+\frac{1}{2} a_2a_5.
$$
This change of variables is easily inverted, as it is lower triangular. Thus, it defines global coordinates $(x_1,\ldots,x_6)$ (i.e., a diffeomorphism) on the unique simply connected Carnot group $\mathfrak{G}\simeq\R^6$, with a Lie algebra isomorphic to $\mathrm{Lie}(g_1,\ldots,g_6)$, if we treat $a_i$ as exponential coordinates on $\mathfrak{G}$.

Let us compute the formulas for multiplication $*$ on $\mathfrak{G}$ in coordinates $x$. This can be done as follows: if $x=\exp\left(\sum_{i=1}^6 a_ig_i\right)$, $y=\exp\left(\sum_{i=1}^6 b_ig_i\right)$ and $z=x*y=\exp\left(\sum_{i=1}^6 c_ig_i\right)$, then
$$
\sum_{i=1}^6 c_ig_i = \log \left(\exp\left(\sum_{i=1}^6 a_ig_i\right)\exp\left(\sum_{i=1}^6 b_ig_i\right)\right),
$$
So, let us compute the right-hand side using the Dynkin–Baker–Campbell–Hausdorff formula\footnote{Due to the nilpotency of the group $\mathfrak{G}$, this formula will only have a finite number of terms.}. As a result, we obtain an expression for $c_i$ in terms of $a_i$ and $b_i$, which in coordinates $x,y,z$ has the form
$$
z = x * y =
\begin{pmatrix}
	x_1+y_1\\
	x_2+y_2\\
	x_3+y_3-x_2y_1\\
	x_4+y_4-x_3y_1+\frac{1}{2} x_2 y_1^2\\
	x_5+y_5+x_1 y_4-x_1x_3y_1+\frac13 x_2y_1^3 -\frac12 x_3y_1^2+\frac12x_1x_2y_1^2\\
	x_6+y_6-x_3y_4+x_2y_5+\frac12 x_3^2 y_1-\frac12 x_2 x_3 y_1^2+\frac16x_2^2 y_1^3
\end{pmatrix}
$$
Thus, $\mathfrak{G}$ with the constructed multiplication operation $*$ becomes a Carnot group of step $5$ and growth vector $(2,1,1,1,1)$. The coordinates $(x_1,\ldots,x_6)$ are naturally called Fuller coordinates on the group $\mathfrak{G}$. Note that the formulas for multiplication $*$ and for left-invariant vector fields in the classical exponential coordinates on $\mathfrak{G}$ turn out to be much more cumbersome. For example, $c_6 = a_6+b_6 + P(a,b)$, where $P$ is a polynomial with 26 terms.

\begin{remark}
	The ``Fuller'' symmetry group on $\mathfrak{G}$ has the form
	$$
	x_1\mapsto\lambda x_1;\quad
	x_2\mapsto\lambda x_2;\quad
	x_3\mapsto\lambda^2x_3;\quad
	x_4\mapsto\lambda^3x_4;\quad
	x_5\mapsto\lambda^4x_5;\quad
	x_6\mapsto\lambda^5 x_6.
	$$
	and coincides with the classical group of dilations arising from the gradation of the Lie algebra of the Carnot group $\mathfrak{G}$. In addition, there is a group of shifts $x_4\mapsto x_4+c_1$, $x_5\mapsto x_5+c_2$, $x_6\mapsto x_6+c_3$, which preserves the vector fields $g_i$.
\end{remark}

\subsection{Left-invariant sub-Finsler structure}
\label{subsec: Carnot sub Finsler}

Let us introduce a left-invariant sub-Finsler structure on $\mathfrak{G}$. Let the Lie algebra $\mathfrak{g}=T_{\id}\mathfrak{G}$ in the basis $g_i(\id)$ be given by coordinates $\eta_i$, $\mathfrak{g}=\{\sum_{i=1}^6\eta_ig_i(\id)\}\simeq\R^6$. Then let's set for $\eta\in\mathfrak{g}$
\begin{equation}
	\label{eq: subnorm Carnot Fuller}
	\|\eta\| = \begin{cases}
		\max\{|\eta_1|,|\eta_2|\},&\text{if }\eta_3=\eta_4=\eta_5=\eta_6=0;\\
		+\infty,&\text{otherwise.}
	\end{cases}
\end{equation}
For an arbitrary vector $\xi\in T_x\mathfrak{G}$, let's define the left-invariant norm as usual by shifting to the identity: $\eta=dL_x^{-1}[\xi]\in\mathfrak{g}$ and $\|\xi\|\stackrel{def}{=}\|\eta\|$. Equivalently: if $\xi = \sum_{i=1}^6 \eta_ig_i(x)$, then $\|\xi\|$ is given by formulas~\eqref{eq: subnorm Carnot Fuller}.

\begin{theorem}
	\label{thm: carnot chattering}
	On the Carnot group $\mathfrak{G}$ with the left-invariant sub-Finsler structure~\eqref{eq: subnorm Carnot Fuller}, there exist pairs of points such that they are connected by a unique shortest path, and it exhibits chattering.
\end{theorem}

An explicit construction of such pairs of points is given in the proof of the theorem.

\medskip

Let us note one more circumstance. By Pansu's theorem (see~\cite{Pansu}), sub-Finsler Carnot groups naturally arise as Gromov–Hausdorff limits of rescaled Cayley graphs of finitely generated virtually nilpotent discrete groups. In connection with this theorem, a natural question arises:

\begin{question}
	What is the structure of geodesics on the Cayley graph of the following discrete nilpotent finitely generated group $G\subset \mathfrak{G}$? Let $\sigma_1=(1,0,0,0,0,0)\in\mathfrak{G}$ and $\sigma_2=(0,1,0,0,0,0)\in\mathfrak{G}$ and set $G=\langle\kappa_1,\kappa_2\rangle\subset \mathfrak{G}$, where $\kappa_1=(1,1,0,0,0,0)=\sigma_1\sigma_2$ and $\kappa_2=(1,-1,0,0,0,0)=\sigma_1\sigma_2^{-1}$. From the multiplication formulas on $\mathfrak{G}$ it easily follows that the discrete group $G_1=\langle\sigma_1,\sigma_2\rangle\supset G$ is a lattice (and, hence, the group $G$ itself), since $G_1=\{(x_1,x_2,x_3,\frac12x_4,\frac16x_5,\frac16x_6)|x_i\in\mathbb{Z}\}$. The elements $\sigma_i$ satisfy the following non-trivial commutation relations (where $[x,y]=xyx^{-1}y^{-1}$):
	$$
	[\sigma_1,\sigma_2] = \sigma_3 = \left(0,0,1,\frac12,\frac16,-\frac16\right);\quad
	[\sigma_1,\sigma_3] = \sigma_4 = \left(0,0,0,1,1,-\frac12\right);
	$$
	$$
	[\sigma_1,\sigma_4] = \sigma_5 = (0,0,0,0,1,0);\quad
	[\sigma_2,\sigma_3] = \sigma_6 = \left(0,0,0,0,0,\frac16\right);
	$$
	$$
	[\sigma_2,\sigma_4] = \sigma_6^6;\quad
	[\sigma_2,\sigma_5] = \sigma_6^6;\quad
	[\sigma_3,\sigma_4] = \sigma_6^{-6}.
	$$
	The remaining commutators of $\sigma_i$ are trivial. The corresponding Gromov--Hausdorff limit coincides with the Carnot group $\mathfrak{G}$ equipped with left-invariant sub-Finsler structure~\eqref{eq: subnorm Carnot Fuller} (see~\cite[(23)]{Pansu}).
\end{question}

\begin{proof}[Proof of Theorem~\ref{thm: carnot chattering}.]
	
	Consider the map $\pi$ from $\mathfrak{G}$ to the space $M=\R^4=\{q=(x,y,z,w)\}$ from paragraph~\ref{sec: subf example}, given by the formulas
	$$
	x=-x_3;\quad y=x_2;\quad z=x_6;\quad w=x_1;
	$$
	that is
	$$
	\pi(x_1,\ldots,x_6) = (-x_3,x_2,x_6,x_1).
	$$
	It is easy to verify that $d\pi[g_i]=f_i$ for every $i$.
	
	The proof of the theorem is based on two key properties of the map $\pi$. First, if $\xi\in T\mathfrak{G}$ and $\|\xi\|<\infty$, then $\|\xi\|=\|d\pi[\xi]\|$. Therefore, the map $\pi$ preserves the lengths of horizontal curves, and, consequently, does not increase the sub-Finsler distance between points:
	$$
	\forall x^0,x^1\in \mathfrak{G}\qquad d_{SF}(\pi(x^0),\pi(x^1)) \le d_{SF}(x^0,x^1).
	$$
	
	Second, every horizontal curve on $M$ can be uniquely lifted (up to the choice of the initial point) to a horizontal curve in $\mathfrak{G}$, and their lengths coincide. Indeed, if $q(t)$ is a horizontal (Lipschitz) curve on $M$, then there exist $u_1(t),u_2(t)\in L_\infty$ such that $\dot q = u_1f_1(q) + u_2f_2(q)$ for almost all $t$. Suppose some horizontal (Lipschitz) curve $x(t)$ on $\mathfrak{G}$ projects to $q(t)$, $\pi(x(t))=q(t)$. Then $\dot x = v_1g_1(x) + v_2g_2(x)$ for some $v_1(t),v_2(t)\in L_\infty$. Therefore $d\pi(\dot x(t)) = v_1(t)f_1(q(t)) + v_2(t)f_2(q(t))$. Since the vectors $f_1(q)$ and $f_2(q)$ are linearly independent at any point $q\in M$, we get that $v_1(t)=u_1(t)$ and $v_2(t)=u_2(t)$ for almost all $t$. Therefore, if an initial point $x(0)\in\mathfrak{G}$ is given such that $\pi(x(0))=q(0)$, then the ODE $\dot x = u_1(t)g_1(x) + u_2(t)g_2(x)$ has a unique solution on any time interval (due to the completeness of $u_1g_1+u_2g_2$). The lengths of the curves coincide, since $\|\dot x\| = \max\{|u_1(t)|,|u_2(t)|\}$ and $\|\dot q\| = \max\{|v_1(t)|,|v_2(t)|\}$.
	
	Combining both properties of the map $\pi$, we get that for every shortest path $q(t)\in M$, $t\in[0;1]$ on $M$, there exists a unique horizontal lift $x(t)\in \mathfrak{G}$, $\pi(x(t))=q(t)$, up to the choice of the initial point $x(0)$. Therefore
	$$
	d_{SF}(q(0),q(1)) \le d_{SF}(x(0),x(1)) \le
	\mathrm{length}(x(t)) = \mathrm{length}(q(t)) =
	d_{SF}(q(0),q(1))
	$$
	Thus, the lift of a shortest path on $M$ is a shortest path on $\mathfrak{G}$. It remains to apply Theorem~\ref{thm: subf chattering}.
\end{proof}

\subsection{Left-invariant Finsler structure}
\label{subsec: Carnot Finsler}

Let us now move on to the left-invariant Finsler problem on $\mathfrak{G}$. Define the unit ball of the left-invariant Finsler norm on $\mathfrak{G}$ as a hyperoctahedron with 12 vertices:
$$
B(x) = \conv(\pm g_1(x)\pm g_2(x);\pm g_3(x);\pm g_4(x);\pm g_5(x);\pm g_6(x)),
$$
that is, for $\xi=(\xi_1,\ldots,\xi_6)\in T_x\mathfrak{G}$ we have the left-invariant norm
\begin{equation}
	\label{eq: norm Finsler Carnot}
	\|\xi\| = \max(|\xi_1|,|\xi_2|) + |\xi_3+x_2\xi_1| + |\xi_4+x_3\xi_1| +
	|\xi_5-x_1\xi_4| + |\xi_6 + \tfrac12x_3^2\xi_1 + (x_1x_2+x_3)\xi_4 - x_2\xi_5|.
\end{equation}

\begin{theorem}
	On the Carnot group $\mathfrak{G}$ with the left-invariant Finsler norm~\eqref{eq: norm Finsler Carnot}, there exist pairs of points such that they are connected by a unique shortest path, and it exhibits chattering.
\end{theorem}

An explicit construction of such pairs of points is given in the proof of the theorem.

\begin{proof}
	The proof essentially repeats the ideas from the proof of Theorem~\ref{thm: subf chattering}. Consider the initial points
	$$
	x^0=(w_0,y_0,-x_0,0,0,z_0);\quad
	x^1=(w_1,-y_1,-x_1,A,B,z_1).
	$$
	The constants $x_{0,1},y_{0,1}$ are arbitrary, the constants $w_{0,1}$ must satisfy the inequality $w_1-w_0\ge T_F(x_0,y_0) + T_F(x_1,-y_1)$, the constants $z_{0,1}$ must satisfy the equality $z_1-z_0=J_F(x_0,y_0) + J_F(x_1,-y_1)$, and the constants $A$ and $B$ will be defined below.
	
	Let $(x_F(t), y_F(t))$ with control $u_F(t)$ be the optimal solution in the modified Fuller problem~\eqref{problem: Fuller finite time} with $t_1=w_1-w_0$. This solution exists and is unique by Lemma~\ref{lm: Fuller finite time}.
	
	Define the curve $\hat x(t)$ on $\mathfrak{G}$:
	$$
	\hat x_1(t) = w_0+t;\
	\hat x_2(t) = y_F(t);\
	\hat x_3(t) = -x_F(t);
	$$
	$$
	\hat x_4(t) = -\int_0^t \hat x_3(s)\,ds;\
	\hat x_5(t) = -\int_0^t \hat x_1(s)\hat x_3(s)\,ds;\
	\hat x_6(t) = z_0+\tfrac12\int_0^t \hat x_3^2(s)\,ds.
	$$
	Direct substitution shows that $\dot{\hat x} = g_1(\hat x) + u_F(t)g_2(\hat x)$. Consequently, according to~\eqref{eq: norm Finsler Carnot} we have $\|\dot{\hat x}\| = 1$. Thus, $\hat x(t)$ is a naturally parameterized curve. Let us set $A=\hat x_4(t_1)$, $B=\hat x_5(t_1)$. Then for $t\in[0;t_1]$ the curve $\hat x(t)$ connects the points $x^0$ and $x^1$. Therefore, the distance between points $x^0$ and $x^1$ does not exceed $t_1$.
	
	Let us show that if some curve $x(t)\in\mathfrak{G}$ connects $x^0$ and $x^1$ and has length no greater than $t_1$, then it has length exactly $t_1$ and coincides with $\hat x(t)$. Indeed, without loss of generality, assume that the curve $x(t)$ is defined for $t\in[0;T]$, $T\le t_1$, and is parameterized naturally, $\|\dot x\|=1$. Then, according to the definition of the norm~\eqref{eq: norm Finsler Carnot}, we have
	$$
	t_1 = w_1-w_0 = x_1(T) - x_1(0) = \int_0^T\dot x_1(t)\,dt \le \int_0^T \|\dot x(t)\|\,dt = T
	$$
	Therefore $T=t_1$, $\dot x_1(t)\equiv 1$, $|\dot x_2|\le 1$. Moreover, since the first term in~\eqref{eq: norm Finsler Carnot} equals $1$ for $\dot x(t)$, the remaining terms must equal $0$, i.e.,
	$$
	\dot x_3 = -x_2\dot x_1;\
	\dot x_4 = -x_3\dot x_1;\
	\dot x_5 = -x_1\dot x_4;\
	\dot x_6 = -\tfrac12 x_3^2 \dot x_1 - (x_1x_2+x_3)\dot x_4 + x_2\dot x_5.
	$$
	Direct substitution of $\dot x_1\equiv1$ gives $\dot x = g_1(x) + u g_2(x)$, where $u=\dot x_2\in[-1;1]$. In particular, the pair $(-x_3(t),x_2(t))$ satisfies the control system in the Fuller problem, and
	$$
	x_6(t_1)-x_6(0)=\tfrac12\int_0^{t_1} x_3^2(t)\,dt=z_1-z_0=J_F(x_0,y_0) + J_F(x_1,-y_1).
	$$
	Therefore, the curve $(-x_3(t),x_2(t))$ is optimal in the modified Fuller problem~\eqref{problem: Fuller finite time} by Lemma~\ref{lm: Fuller finite time}. That is, $x_3(t)=-x_F(t)$ and $x_2(t)=y_F(t)$. Hence $x_2(t)\equiv \hat x_2(t)$, $x_3(t)\equiv \hat x_3(t)$. Furthermore, $x_1(t)\equiv \hat x_1(t)=w_0+t$, and the equalities $x_i(t)\equiv \hat x_i(t)$ for $i\ge 4$ follow from the integral formulas for them.
	
	Thus, with an appropriate choice of constants $A$ and $B$, the curve $\hat x(t)$ is the unique shortest path connecting points $x^0$ and $x^1$, and this curve exhibits chattering if $(x_0,y_0)\ne (0,0)$ or $(x_1,y_1)\ne(0,0)$.
\end{proof}

\section{Theorem on the generality of chattering regimes in polyhedral Finsler and sub-Finsler structures}
\label{sec: general theorem}

In this section, we will use the direct definition~\ref{defn: polyhedral1} of the sub-Finsler norm.

\begin{definition}
	We say that vectors $\xi_1,\ldots,\xi_k$ form an exposed face in the polyhedron $\conv(\xi_1,\ldots,\xi_N)$, $N > k$, if they are affinely independent (i.e., $\dim\aff(\xi_1,\ldots,\xi_k)=k-1$) and the affine subspace $\aff(\xi_1,\ldots,\xi_k)$ does not intersect the convex hull of the remaining points:
	$$
	\aff(\xi_1,\ldots,\xi_k)\ \cap\ \conv(\xi_{k+1},\ldots,\xi_N) = \emptyset
	$$
\end{definition}

Obviously, in the case $k=1$ we call the face a vertex, and in the case $k=2$ -- an edge.

\medskip

Let us now formulate a sufficient condition for the existence of chattering extremals in the sub-Finsler problem. Let a polyhedral sub-Finsler norm $\|\cdot\|_1$ be given on the manifold $M$ in the sense of the direct definition~\ref{defn: polyhedral1}: $B_1(x)=\conv(f_i(x))$, where $f_i$ are smooth vector fields. Let $x_0\in M$. Select some exposed edge in $B_1(x_0)$, for example, $\conv(f_1(x_0),f_2(x_0))$.

\begin{definition}
	Let $x_0\in M$. We will say that $p_0\in T^*_{x_0}M$ is a \textit{Fuller} covector with respect to the exposed edge $\conv(f_1(x_0),f_2(x_0))$, if
	\begin{enumerate}
		\item the covector $p_0$ strictly separates the affine line of the edge from the convex hull of the remaining vertices
		$$
		\sup_{\xi\in\conv_{i\ge 3}f_i(x_0)} \langle p_0,\xi\rangle
		<
		\inf_{\xi\in\aff(f_1(x_0),f_2(x_0))} \langle p_0,\xi\rangle;
		$$
		
		\item the 7 vector fields
		$$
		f_2-f_1;\ [f_1,f_2];\ [f_1,[f_1,f_2]];\ [f_2,[f_1,f_2]];
		$$
		$$
		[f_1,[f_1,[f_1,f_2]]];\ [f_2,[f_1,[f_1,f_2]]]=[f_1,[f_2,[f_1,f_2]]];\ [f_2,[f_2,[f_1,f_2]]]
		$$
		are linearly independent at the point $x_0$;
		
		\item the covector $p_0$ is orthogonal to the specified 7 vector fields at the point $x_0$.
	\end{enumerate}
\end{definition}

\begin{theorem}
	\label{thm: chattering ubiquity}
	Let $x_0\in M$ and $\conv(f_1(x_0),f_2(x_0))$ be an exposed edge in $B_1(x_0)$. Denote $f=f_1+f_2$, $g=f_2-f_1$ and set
	$$
	\begin{array}{rr}
		\alpha   = [f,[f,[f,[f,g]]]](x_0);&
		\beta    = [g,[f,[f,[f,g]]]](x_0);\\
		\gamma   = [f,[g,[f,[f,g]]]](x_0);&
		\delta   = [g,[g,[f,[f,g]]]](x_0);\\
		\epsilon = [f,[g,[g,[f,g]]]](x_0);&
		\zeta    = [g,[g,[g,[f,g]]]](x_0).
	\end{array}
	$$
	If there exists such a Fuller covector $\hat p\in T^*_{x_0}M$ that
	$$
	\langle \hat p,\beta\rangle<0;\
	\langle \hat p,\alpha\rangle=\langle \hat p,\gamma\rangle = \langle \hat p,\delta\rangle = \langle \hat p,\epsilon\rangle = \langle \hat p,\zeta\rangle = 0;
	$$
	then for every Fuller covector $p_0\in T^*_{x_0}M$ from some neighborhood of $\hat p$, a one-parameter family of chattering extremals\footnote{Recall that an extremal $H=\sum_i u_i\langle p,f_i(x)\rangle$ exhibits the chattering phenomenon at some point if there is a countable number of control switches on the extremal in a neighborhood of this point.} with controls $(1,0,\ldots,0)$ and $(0,1,0,\ldots,0)$ enters the point $(x_0,p_0)\in T^*M$, and a one-parameter family of chattering extremals with the same controls exits.
\end{theorem}

\begin{proof}
	The proof of this theorem is essentially simple and is a reduction to the known Zelikin–Borisov theorem on the generality of the Fuller phenomenon in optimal control problems. Let's perform this reduction. So, sub-Riemannian shortest paths are solutions to the time-optimal problem~\eqref{eq: time minimiztion problem}, and the Pontryagin Maximum Principle has the form~\eqref{eq: H for time minimiztion}.
	
	If the covector $\hat p$ separates the affine line $\aff(f_1(x_0),f_2(x_0))$ from the convex hull of the remaining vertices, then for any $p\in T_xM$ from a neighborhood of $\hat p$, the maximum of the Pontryagin function $H=\sum_i u_i\langle p,f_i(x)\rangle$ is attained either at the vertex $f_1(x)$, or at the vertex $f_2(x)$, or on the entire segment $\conv(f_1(x),f_2(x))$. Therefore, in a neighborhood of $\hat p$, the extremals coincide with the solutions of the system with the simplified Pontryagin function
	$$
	\hat H = u_1\langle p, f_1(x)\rangle + u_2\langle p, f_2(x)\rangle,
	$$
	where $u_1\ge0$, $u_2\ge 0$ and $u_1+u_2=1$. Equivalently
	$$
	\hat H = F + Gu
	$$
	where $F=\frac12\langle p,f\rangle$, $G=\frac12\langle p,g\rangle$ and $u=u_2-u_1\in[-1;1]$.
	
	So, in a neighborhood of the point $\hat p\in T^*_{x_0}M$, the extremals of the original Pontryagin function $H$ and the simplified one $\hat H$ coincide. Let's check that the extremals of $\hat H$ in the neighborhood of $\hat p$ satisfy the result of the theorem. Note also that replacing $\hat p$ with $\lambda\hat p$, $\lambda>0$, changes nothing, so we can assume that $\langle \hat p,\beta\rangle =-8$.
	
	Recall that the Poisson bracket of Hamiltonians homogeneous in momentum is computed as follows: if $h_1$ and $h_2$ are vector fields, then
	$$
	\{\langle p,h_1\rangle,\langle p,h_2\rangle\} = \langle p,[h_1,h_2]\rangle.
	$$
	Therefore, the Poisson brackets of $F$ and $G$ are structured similarly to the commutators of the vector fields $f$ and $g$ (up to division by $2^{r-1}$, where $r$ is the length of the bracket). Therefore, a simple direct check shows that the Pontryagin function $\hat H$ satisfies the conditions of Theorem 4.1 and Remark 4.2 from the book~\cite{ZB}.
\end{proof}

\begin{corollary}
	The set of such points $(x_0,p_0)$ forms a smooth submanifold of codimension 7 in $T^*M$.
\end{corollary}

\begin{remark}
	Note that since the covector $\hat p$ strictly separates the edge $\conv(f_1(x_0),f_2(x_0))$ from the convex hull of the remaining vertices, then $\sup_u H(x_0,\hat p,u)>0$, and, consequently, all the chattering extremals found in Theorem~\ref{thm: chattering ubiquity} are normal.
\end{remark}

\section{Conclusion}
\label{sec: chaos}

The method presented in this article can also be used to obtain Finsler and sub-Finsler shortest paths with another known phenomenon in optimal control, which arises, for example, in the following problem (see~\cite{ZLH}):
$$
\begin{array}{c}
	\int_0^\infty (x_1^2 + x_2^2)\,dt\to\min;\\
	\ddot x_1=u_1,\ \ddot x_2=u_2;\quad(u_1,u_2)\in U;\\
	x_1(0)=x_{10},\ x_2(0)=x_{20},\ \dot x_1(0)=y_{10},\ \dot x_2(0)=y_{20};
\end{array}
$$
where the set of admissible controls $U$ is an equilateral triangle centered at the origin. In this problem, all optimal trajectories (except the identically zero one) reach the origin with chattering. However, the order of control alternation on different optimal trajectories can be completely different. Indeed, in the classical Fuller problem, the set of controls is a segment, and chattering involves simple alternation between the segment's endpoints. In the problem mentioned above, the set of controls $U$ is a triangle, and the order of vertex alternation during chattering on different optimal trajectories can be completely different. In this problem, the order of vertex alternation between control switches is determined by a topological Markov chain on some non-trivial graph. Finsler and sub-Finsler manifolds can also contain similar shortest paths. The simplest examples are constructed analogously to the examples in sections~\ref{sec: subf example} and~\ref{sec: finsl example} on the space $\R^6$ (it is necessary to add two coordinates to the original phase space $\R^4$ — the value of the functional and the travel time).


\begin{thebibliography}{99}

\bibitem{LokutsievskiyCT2} A.A. Ardentov, L.V. Lokutsievskiy, Yu.L. Sachkov, "Extremals for a series of sub-Finsler problems with 2-dimensional control via convex trigonometry", ESAIM: COCV, 27 (2021), 32 , 52 pp., arXiv: 2004.10194 

\bibitem{ArdentovSachkov} A. Ardentov, Yu. Sachkov "Sub-Finsler Geodesics on the Cartan Group", Journal of Dynamical and Control Systems, Vol. 25, 2019.

\bibitem{Bao} D. Bao, S.-S. Chern, Z. Shen,  "An Introduction to Riemann-Finsler Geometry", Springer, 2000.

\bibitem{BBLDS} D. Barilari, U. Boscain, E. Le Donne, and M. Sigalotti, "Sub-Finsler structures from the time-optimal control viewpoint for some nilpotent distributions", J. Dyn. Control Syst. 23 (2017), no. 3, 547–57

\bibitem{Fuller} A.T. Fuller "Relay Control Systems Optimized for Various Performance Criteria", Automatic and Remote Control, Proceedings of the First IFAC Congress, Moscow, 1960.

\bibitem{LeDonne} E. Le Donne, "Metric Lie Groups, Carnot-Carathéodory spaces from the homogeneous viewpoint", Springer Nature, 2025

\bibitem{LokutsievskiyCT}  L.V. Lokutsievskiy, "Convex trigonometry with applications to sub-Finsler geometry", Sb. Math., 210:8 (2019), 1179–1205, arXiv: 1807.08155 

\bibitem{Pansu} Pansu. "Croissance des boules et des g\'eod\'esiques ferm\'ees dans les nilvari\'et\'es", Ergodic Theory Dynam. Systems, 3(3):415–445, 1983.

\bibitem{ZB}  M. Zelikin , V. Borisov, "Theory of Chattering Control with applications to Astronautics, Robotics, Economics, and Engineering", 1994

\bibitem{ZLH} M. I. Zelikin, L. V. Lokutsievskii, R. Hildebrand, “Typicality of Chaotic Fractal Behavior of Integral Vortices in Hamiltonian Systems with Discontinuous Right Hand Side”, Journal of Mathematical Sciences, 221:1 (2017), 1–136

\bibitem{ZMH} 	M. I. Zelikin, N. B. Melnikov, R. Hildebrand, “The Topological Structure of the Phase Portrait of a Typical Fiber of Optimal Synthesis for Chattering Problems”, Differential equations. Certain mathematical problems of optimal control, Collected papers, Trudy Mat. Inst. Steklova, 233, Nauka, MAIK «Nauka/Inteperiodika», M., 2001, 125–152; Proc. Steklov Inst. Math., 233 (2001), 116–142

\end{thebibliography}
\end{document}